\documentclass[a4paper,12pt,oneside]{article}

\usepackage{mathtools}
\usepackage{tikz}
\usetikzlibrary{shapes,arrows,positioning} 
\tikzset{
    >=stealth',
    punkt/.style={
           rectangle,
           rounded corners,
           draw=black, very thick, 
           text width=8em,
           minimum height=2em,
           text centered, fill=white, drop shadow},
    punkta/.style={
           rectangle,
           rounded corners,
           draw=black, very thick,
           text width=10em,
           minimum height=2em,
           text centered, fill=white, drop shadow},
    punktaka/.style={
           rectangle,
           rounded corners,
           draw=black, very thick,
           text width=14em,
           minimum height=2em,
           text centered, fill=white, drop shadow},       
    punktaa/.style={
           rectangle,
           rounded corners,
           draw=black, very thick,
           text width=15em,
           minimum height=2em,
           text centered, fill=white, drop shadow},
    punktaaa/.style={
           rectangle,
           rounded corners,
           draw=black, very thick,
           text width=10em,
           minimum height=2em,
           text centered, fill=white, drop shadow},
    pil/.style={
           ->,
           thick,
           shorten <=2pt,
           shorten >=2pt,}
}

\usepackage[framemethod=tikz]{mdframed}
\usepackage{soul}
\usepackage{transparent}
\usetikzlibrary{arrows, decorations.markings}
\usetikzlibrary{automata,positioning}

\usepackage[T1]{fontenc}
\usepackage{tikz}
\usetikzlibrary{shadings,shadows,shapes.arrows}

\usetikzlibrary{shapes,snakes}
\usetikzlibrary{matrix}

\usepackage{fancybox}

\usepackage{etex}
\usepackage{tcolorbox}
\usepackage{color}
\usepackage{fancybox,framed}

\usepackage[framemethod=tikz]{mdframed}
\usepackage{lipsum}
\usepackage{amssymb}% http://ctan.org/pkg/amssymb
\usepackage{pifont}
\definecolor{mycolor}{rgb}{0.122, 0.435, 0.698}
\usepackage{emptypage}

\usepackage{afterpage}

\usepackage{epigraph}
\usepackage{url}
\usepackage{fancyhdr}
\usepackage{enumitem}
\usepackage{wrapfig}
\usepackage[toc,page]{appendix}
\usepackage{tikz}
\usepackage{fp}
\usepackage{pgf}
\usepackage{xifthen}
\newcommand{\warsawApp}[2] %%%Intersection of grid of side \sqrt{2}/2^(#1-1) with the Warsaw circle if #2=1 OR \sqrt{2}/2^{#1}-app with of the Warsaw circle if #2=anything else%%%
{
\FPeval{\points}{4-((#1)/2)}
\begin{tikzpicture}[scale=5,domain=0:1] 
\tikzstyle{every node}=[circle, draw, fill=black!50,
                        inner sep=0pt, minimum width=\points pt]
\FPeval{\step}{1/2^((#1))} 
\FPeval{\twostep}{1/2^(2*((#1)))}
\FPeval{\twostepone}{1/2^(2*((#1))-1)}
\FPeval{\twosteptwo}{1/2^(2*((#1))-2)}
\FPeval{\stepone}{1/2^(((#1))-1)}
\FPeval{\B}{(1/2)+(1/2^(((#1))-1))}
\FPeval{\A}{(1/2)-(1/2^(((#1))-1))}
\FPeval{\C}{(1/2)+(1/2^(((#1))))}
\FPeval{\D}{1-(1/2^((#1)))}
\FPeval{\E}{1/2+(3/2^((#1)))}
\FPeval{\kminustwo}{((#1))-1}
\pgfmathsetmacro{\puntosa}{int((((#1))-1)/2)}
\pgfmathsetmacro{\puntosb}{int((((#1))-2)/2)}
\ifthenelse{#2=1}{
								\draw[step=\stepone,gray,very thin] (-0.1,-0.1) grid (1.1,1.1);
							  }{}	
\draw (0,1)--(0,0)--(1,0)--(1,1)--(1/2,1)--(1/2,1/2)--(1/4,1/2)--(1/4,1)--(1/8,1)--
(1/8,1/2)--(1/16,1/2)--(1/16,1)--(1/32,1)--(1/32,1/2)--(1/64,1/2)--(1/64,1);
\draw
\foreach \x in {0,\stepone,...,1} {(\x,0) node {} }

\foreach \x in {0,\stepone,...,\A} {(0,\x) node {} }
\foreach \x in {0,\stepone,...,\A} {(1,\x) node {} }

\foreach \l in {0,1,...,\kminustwo} {\foreach \m in {0.5,\B,...,1} {(1/2^\l,\m) node {} }}
\foreach \m in {0.5,\B,...,1} {(0,\m) node {} };

\foreach \x in {1,...,\puntosa} {
													\pgfmathsetmacro{\u}{1/2^(2*\x-1)}
													\pgfmathsetmacro{\v}{1/2^(2*\x-1)+\stepone}
													\pgfmathsetmacro{\w}{1/2^(2*\x-2)}
													\foreach \z in {\u,\v,...,\w} {																								\draw node at (\z,1) {} ;
																								}
										          }
										          
\ifthenelse{#1 > 3}{
                             \foreach \x in {1,...,\puntosb} {
													                          \pgfmathsetmacro{\u}{1/2^(2*\x)}
													                          \pgfmathsetmacro{\v}{1/2^(2*\x)+\stepone}
													                          \pgfmathsetmacro{\w}{1/2^(2*\x-1)}
													                          \foreach \z in {\u,\v,...,\w} {
																								                          \draw node at (\z,0.5) {} ;
																								                         }
										                                     }
							   }{}			                               			
\ifthenelse{#2=0}{
							\foreach \m in {\C,\E,...,\D} {\draw node at (\step,\m) {} ;}	
							}{}
\end{tikzpicture}}
\newcommand{\warsawPol}[2] %%%\sqrt{2}/2^{#1}-pol with nodes of size #2 of the Warsaw circle%%%
{
\FPeval{\points}{4-((#1)/2)}
\begin{tikzpicture}[scale=#2,domain=0:1] 
\tikzstyle{every node}=[circle, draw, fill=black!50,
                        inner sep=0pt, minimum width=\points pt]
\FPeval{\step}{1/2^((#1))} 
\FPeval{\twostep}{1/2^(2*((#1)))}
\FPeval{\twostepone}{1/2^(2*((#1))-1)}
\FPeval{\twosteptwo}{1/2^(2*((#1))-2)}
\FPeval{\stepone}{1/2^(((#1))-1)}
\FPeval{\steptwo}{1/2^(((#1))-2)}
\FPeval{\B}{(1/2)+(1/2^(((#1))-1))}
\FPeval{\BB}{(1/2)+(1/2^(((#1))-2))}
\FPeval{\A}{(1/2)-(1/2^(((#1))-1))}
\FPeval{\C}{(1/2)+(1/2^(((#1))))}
\FPeval{\D}{1-(1/2^((#1)))}
\FPeval{\E}{1/2+(3/2^((#1)))}
\FPeval{\F}{1-(1/2^(((#1))-1))}
\FPeval{\kminustwo}{((#1))-1}
\pgfmathsetmacro{\puntosa}{int((((#1))-1)/2)}
\pgfmathsetmacro{\puntosb}{int((((#1))-2)/2)}
\pgfmathsetmacro{\puntosc}{int((((#1))+1)/2)}
\draw (0,1)--(0,0)--(1,0)--(1,1);
\foreach \x in {1,...,\puntosa}{\pgfmathsetmacro{\u}{1/2^(2*\x-1)}
											     \pgfmathsetmacro{\w}{1/2^(2*\x-2)}
												  \draw (\u,1)--(\w,1);
											   }
\foreach \x in {1,...,\puntosc}{\pgfmathsetmacro{\u}{1/2^\x}										   
												  \draw (\u,1)--(\u,0.5);
											   }
\ifthenelse{#1 > 3}{
\foreach \x in {1,...,\puntosb} {\pgfmathsetmacro{\u}{1/2^(2*\x)}
													\pgfmathsetmacro{\w}{1/2^(2*\x-1)}
													\draw (\u,0.5)--(\w,0.5);
										        }	
							    }{}
\draw (\steptwo,0.5)--(\steptwo,1);	
\draw (\stepone,0.5)--(\stepone,1);
\foreach \x in {0.5,\B,...,1}{\draw (\stepone,\x)--(\steptwo,\x);}	
\draw (0,0.5)--(\stepone,0.5);
\draw (0,1)--(\stepone,1);		
\draw (\step,\C)--(\step,\D);
\ifthenelse{#1 > 3}{
\foreach \m in {\C,\E,...,\F}  {%%Little rombos%%
													\pgfmathsetmacro{\mstepone}{\m+\stepone}
													\pgfmathsetmacro{\mstep}{\m+\step}
													
													\draw [fill=gray, fill opacity=0.85] (\step,\m)--(0,\mstep)--
													(\step,\mstepone)--(\stepone,\mstep)--(\step,\m);																					
									          }	
							   }{
							     \pgfmathsetmacro{\mstepone}{\C+\stepone}
								 \pgfmathsetmacro{\mstep}{\C+\step}
													
									\draw [fill=gray, fill opacity=0.85] (\step,\C)--(0,\mstep)--
													(\step,\mstepone)--(\stepone,\mstep)--(\step,\C);			
								}		
\draw (0,0.5)--(\step,\C)--(\stepone,0.5);	
\draw (0,1)--(\step,\D)--(\stepone,1);
\foreach \x in {\B,\BB,...,\F}{\draw[thin, densely dotted] (0,\x)--(\stepone,\x);}
	 							          	
\draw [fill=gray, fill opacity=0.5]
       (0,0.5) -- (\stepone,0.5) -- (\stepone,1) -- (0,1)--(0,0.5);					       											
\draw
\foreach \x in {0,\stepone,...,1} {(\x,0) node {} }

\foreach \x in {0,\stepone,...,\A} {(0,\x) node {} }
\foreach \x in {0,\stepone,...,\A} {(1,\x) node {} }

\foreach \l in {0,1,...,\kminustwo} {\foreach \m in {0.5,\B,...,1} {(1/2^\l,\m) node {} }}
\foreach \m in {0.5,\B,...,1} {(0,\m) node {} };

\foreach \x in {1,...,\puntosa} {
													\pgfmathsetmacro{\u}{1/2^(2*\x-1)}
													\pgfmathsetmacro{\v}{1/2^(2*\x-1)+\stepone}
													\pgfmathsetmacro{\w}{1/2^(2*\x-2)}
													\foreach \z in {\u,\v,...,\w} {																								\draw node at (\z,1) {} ;
																								}
										          }
\ifthenelse{#1 > 3}{										          
\foreach \x in {1,...,\puntosb} {
													\pgfmathsetmacro{\u}{1/2^(2*\x)}
													\pgfmathsetmacro{\v}{1/2^(2*\x)+\stepone}
													\pgfmathsetmacro{\w}{1/2^(2*\x-1)}
													\foreach \z in {\u,\v,...,\w} {																								\draw node at (\z,0.5) {} ;
																								}
										          }		
								}{}	          	
\foreach \m in {\C,\E,...,\D} {\draw node at (\step,\m) {} ;}				          
\end{tikzpicture}}
\newcommand{\warsawTdPol}[2]
{
\FPeval{\points}{4-((#1)/2)}
\begin{tikzpicture}[scale=#2,line join=bevel,x=5,y=5,z=3,rotate =0]
\tikzstyle{every node}=[circle, draw, fill=black!50,
                        inner sep=0pt, minimum width=\points pt]                        
\FPeval{\step}{1/2^((#1))} 
\FPeval{\twostep}{1/2^(2*((#1)))}
\FPeval{\twostepone}{1/2^(2*((#1))-1)}
\FPeval{\twosteptwo}{1/2^(2*((#1))-2)}
\FPeval{\stepone}{1/2^(((#1))-1)}
\FPeval{\steptwo}{1/2^(((#1))-2)}
\FPeval{\B}{(1/2)+(1/2^(((#1))-1))}
\FPeval{\BB}{(1/2)+(1/2^(((#1))-2))}
\FPeval{\A}{(1/2)-(1/2^(((#1))-1))}
\FPeval{\C}{(1/2)+(1/2^(((#1))))}
\FPeval{\D}{1-(1/2^((#1)))}
\FPeval{\H}{1-(1/2^((#1)))-\step}
\FPeval{\E}{1/2+(3/2^((#1)))}
\FPeval{\F}{1-(1/2^(((#1))-1))}
\FPeval{\kminustwo}{((#1))-1}
\pgfmathsetmacro{\puntosa}{int((((#1))-1)/2)}
\pgfmathsetmacro{\puntosb}{int((((#1))-2)/2)}
\pgfmathsetmacro{\puntosc}{int((((#1))+1)/2)}
\FPeval{\alt}{\stepone}
\draw (0,0,1)--(0,0,0)--(1,0,0)--(1,0,1);
\foreach \x in {1,...,\puntosa}{\pgfmathsetmacro{\u}{1/2^(2*\x-1)}
											     \pgfmathsetmacro{\w}{1/2^(2*\x-2)}
												  \draw (\u,0,1)--(\w,0,1);
											   }
\foreach \x in {1,...,\puntosc}{\pgfmathsetmacro{\u}{1/2^\x}										   
												  \draw (\u,0,1)--(\u,0,0.5);
											   }
\ifthenelse{#1 > 3}{
\foreach \x in {1,...,\puntosb} {\pgfmathsetmacro{\u}{1/2^(2*\x)}
													\pgfmathsetmacro{\w}{1/2^(2*\x-1)}
													\draw (\u,0,0.5)--(\w,0,0.5);
										        }	
								}{}		        
\draw (\steptwo,0,0.5)--(\steptwo,0,1);	
\draw (\stepone,0,0.5)--(\stepone,0,1);
\foreach \x in {0.5,\B,...,1}{\draw (\stepone,0,\x)--(\steptwo,0,\x);}	
\draw (0,0,0.5)--(\stepone,0,0.5);
\draw (0,0,1)--(\stepone,0,1);		
\draw
\foreach \x in {0,\stepone,...,1} {(\x,0,0) node {} }

\foreach \x in {0,\stepone,...,\A} {(0,0,\x) node {} }
\foreach \x in {0,\stepone,...,\A} {(1,0,\x) node {} }

\foreach \l in {0,1,...,\kminustwo} {\foreach \m in {0.5,\B,...,1} {(1/2^\l,0,\m) node {} }}
\foreach \m in {0.5,\B,...,1} {(0,0,\m) node {} };

\foreach \x in {1,...,\puntosa} {
													\pgfmathsetmacro{\u}{1/2^(2*\x-1)}
													\pgfmathsetmacro{\v}{1/2^(2*\x-1)+\stepone}
													\pgfmathsetmacro{\w}{1/2^(2*\x-2)}
													\foreach \z in {\u,\v,...,\w} {																								\draw node at (\z,0,1) {} ;
																								}
										          }		
\ifthenelse{#1 > 3}{										          							          
 \foreach \x in {1,...,\puntosb} {
													\pgfmathsetmacro{\u}{1/2^(2*\x)}
													\pgfmathsetmacro{\v}{1/2^(2*\x)+\stepone}
													\pgfmathsetmacro{\w}{1/2^(2*\x-1)}
													\foreach \z in {\u,\v,...,\w} {
																								\draw node at (\z,0,0.5) {} ;
																								}
										          }		
							  }{}		   											          		
\ifthenelse{#1 > 3}{
	\foreach \m in {\C,\E,...,\H} {            								
											  \pgfmathsetmacro{\mpstepone}{\m+\stepone}
											  \pgfmathsetmacro{\mpstep}{\m+\step}
											  \pgfmathsetmacro{\mmstep}{\m-\step}
											  \coordinate (ul) at (0,0,\mpstep);	
											  \coordinate (ur) at (\stepone,0,\mpstep);	
											  \coordinate (dl) at (0,0,\mmstep);		
											  \coordinate (dr) at (\stepone,0,\mmstep);
											  \draw [fill=gray, fill opacity=1] (\step,\alt,\m)--(ul)--(ur)--cycle;						 
						    				  }	
						       }{\draw [fill=gray, fill opacity=1] (\step,\alt,\C)--(0,0,\B)--(\stepone,0,\B)--cycle;}
\draw node at (0,0,\B) {};						       
\ifthenelse{#1 > 3}{							    				  
\foreach \m in {\C,\E,...,\H} {            								
											  \pgfmathsetmacro{\mpstepone}{\m+\stepone}
											  \pgfmathsetmacro{\mpstep}{\m+\step}
											  \pgfmathsetmacro{\mmstep}{\m-\step}
											  \coordinate (ul) at (0,0,\mpstep);	
											  \coordinate (ur) at (\stepone,0,\mpstep);	
											  \coordinate (dl) at (0,0,\mmstep);		
											  \coordinate (dr) at (\stepone,0,\mmstep);
											  \draw [fill=gray, fill opacity=1] (\step,\alt,\m)--(ur)--(\step,\alt,\mpstepone)--cycle;
								            }	
								 }{\draw [fill=gray, fill opacity=1] (\step,\alt,\C)--(\stepone,0,\B)--(\step,\alt,\E)--cycle;}						            		
\foreach \m in {\C,\E,...,\D} {            								
											  \pgfmathsetmacro{\mpstep}{\m+\step}
											  \pgfmathsetmacro{\mmstep}{\m-\step}
											  \coordinate (ul) at (0,0,\mpstep);	
											  \coordinate (ur) at (\stepone,0,\mpstep);	
											  \coordinate (dl) at (0,0,\mmstep);		
											  \coordinate (dr) at (\stepone,0,\mmstep);							 
											  \draw [fill=gray, fill opacity=0.5] (\step,\alt,\m)--(ur)--(dr)--cycle;
										    }								 
\draw [fill=gray, fill opacity=0.5] (\step,\alt,\C)--(0,0,0.5)--(\stepone,0,0.5)--cycle;			
\foreach \m in {0.5,\B,...,1} {\draw node at (\stepone,0,\m) {} ;}
\foreach \m in {\C,\E,...,\D} {\draw node at (\step,\alt,\m) {} ;}
\draw node at (0,0,0.5) {};
\end{tikzpicture}
}

\usepackage[utf8]{inputenc}
\usepackage{color}
\usepackage{import}
\usepackage{setspace}
\usepackage{latexsym,amsmath,amsfonts,amscd,amssymb,amsthm}
\usepackage{graphicx, caption}
\usepackage[all]{xy}
\usepackage{epstopdf}
\DeclareGraphicsExtensions{.pdf,.png,.jpg,.gif, .eps}

\usepackage{amsthm}
\usepackage{thmtools}

\theoremstyle{plain} % default
\newtheorem{teo}{Theorem}
\newtheorem{lem}{Lemma}
\newtheorem{prop}{Proposition}

\theoremstyle{definition}

\theoremstyle{remark}
\newtheorem{obs}{Remark}

\newcommand{\conjunto}[1]{\left\lbrace #1 \right\rbrace}

\newcommand{\cms}{$(X,\textrm{d})$ }

\newcommand{\todon}{n\in\mathbb{N}}

\renewcommand{\epsilon}{\varepsilon}

\newcommand{\dist}[2]{\textrm{d}(#1,#2)}

\newcommand{\diam}{\textrm{diam}}

\usepackage{hyperref}
\hypersetup{
	colorlinks=true,
	linkcolor=blue,
	filecolor=magenta,      
	urlcolor=cyan,
	pdfpagemode=FullScreen,
}
\usepackage[light, math]{iwona}
\usepackage[T1]{fontenc}
\makeindex
\makeatletter
\newcommand{\subjclass}[2][2010]{%
  \let\@oldtitle\@title%
  \gdef\@title{\@oldtitle\footnotetext{#1 \emph{Mathematics Subject Classification.} #2}}%
}
\newcommand{\keywords}[1]{%
  \let\@@oldtitle\@title%
  \gdef\@title{\@@oldtitle\footnotetext{\emph{Key words and phrases.} #1.}}%
}
\makeatother
\title{\bf{On a finite type Multivalued Shape}}
\usepackage{authblk}
\author{Diego Mondéjar}
\date{}                     %% if you don't need date to appear

\begin{document}
\maketitle
\begin{center}
\em Dedicated with respect to Professor J.M.R. Sanjurjo for his 70th birthday
\end{center}
\begin{abstract}
We briefly review the origins and development of Borsuk's Theory of Shapes and the Multivalued Shape of Sanjurjo. We use a construction over metric compacta using hyperspaces to define a finite type version.
\end{abstract} 
\section{Introduction}\label{sec:intro}

Shape Theory is a suitable extension of homotopy theory for topological spaces with bad local properties, where this theory does not give relevant information. The paradigmatic example is the Warsaw circle $\mathcal{W}$ of Figure \ref{warsawcircle}. It can be defined as the graph of function $f(x)=\sin\left(\frac{1}{x}\right)$ in the interval $(0,\frac{2}{\pi}]$ adding its closure (that is, the segment joining $(0,-1)$ and $(0,1)$) and closing the space by any simple (not intersecting itself or the rest of the space) arc joining the points $(0,-1)$ and $(\frac{\pi}{2},1)$.
\begin{figure}
	\begin{center}
		\begin{tikzpicture}[scale=2,domain=0:1,x=4cm]
			\FPeval{\w}{2/(3.1415)}
			\draw[domain=0.004:\w,samples=5000] plot (\x, {sin((1/\x)r)});
			\draw (0,1)--(0,-1.5)--(\w,-1.5)--(\w,{sin((1/\w)r)});
		\end{tikzpicture}
	\end{center}
	\caption{The Warsaw circle.}
	\label{warsawcircle}
\end{figure}
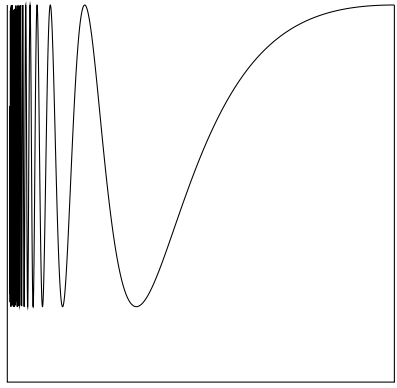 
It is readily seen that the fundamental group of $\mathcal{W}$ is trivial, since there are not continuous maps from $S^1$ to  $\mathcal{W}$ appart from the class of nullhomotopic maps. Moreover, so are all its homology and homotopy groups. But it is also easy to see that $\mathcal{W}$ is not contractible and, for example, it decomposes the plane in two connected components. It is then evident that homotopy theory does not work well for $\mathcal{W}$. One would like to detect this structure as something similar, in some sense, to $S^1$, since both separate the plane in two parts. The problem is that there is a deficiency of morphisms between $S^1$ and $\mathcal{W}$ in homotopy theory because of the bad local behaviour of $\mathcal{W}$ in some points. Shape was initiated in the celebrated paper \cite{Bconcerning} by Karol Borsuk in 1968 to overcome these limitations, defining a new category, containing the same information about well behaved topological spaces, but giving some relevant classification for spaces with bad local properties. The idea is that, no matter how bad the space is, its neighborhoods when it is embedded into a larger space (for example the Hilbert cube $Q$) are not too bad. In our example, it is easy to see in Figure \ref{comparation} that the neighborhoods of $\mathcal{W}$ are annuli, with the same the homotopy type of than the ones in $\mathbb{S}^1$. 
\begin{figure}
\begin{center}
	\captionsetup{width=.8\linewidth}
	\includegraphics[width=13cm]{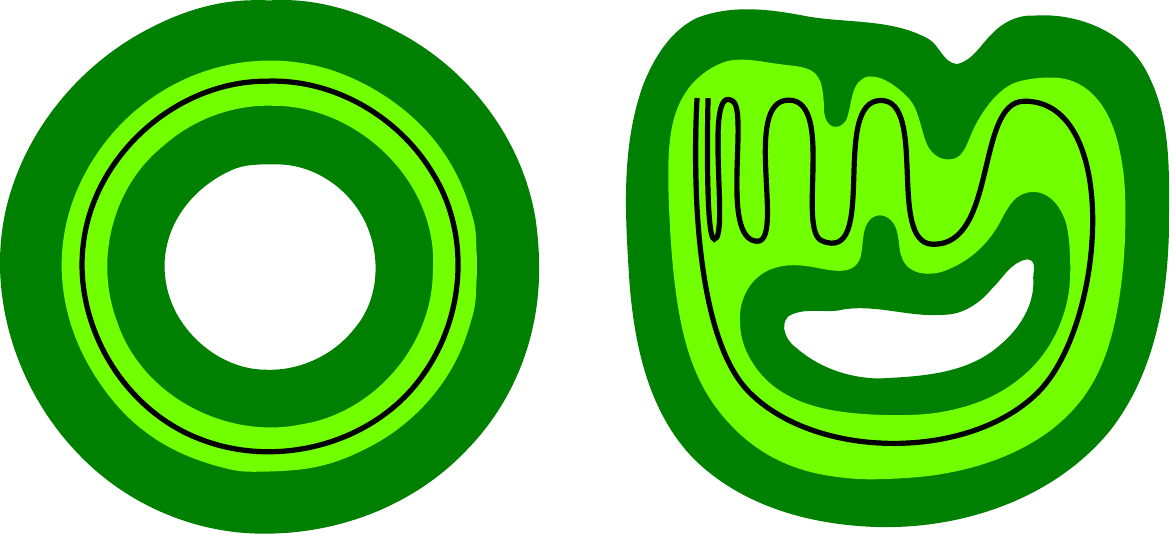}
	\caption{A circle and the warsaw circle compared by its neighborhoods. The first neighborhood is the dark and light colored annuli together and the second is just the light one. It is important to consider the homotopy type of the annuli and the way each annulus is included into the next one.}
\label{comparation}
\end{center}
\end{figure}
Then, we can compare spaces by comparing their neighborhoods and this process can be done for smaller and smaller neighborhoods as a limit. The space $\mathcal{W}$ share some global properties with $\mathbb{S}^1$ since their neighborhoods do. In this way, there are no non-trivial maps from $\mathbb{S}^1$ to $\mathcal{W}$, so we can compare then and detect if they share any topological property.

Specifically, Borsuk defined a new class of morphisms between metric compacta embedded in the Hilbert cube, called fundamental sequences, as sequences of continuous maps $f_n:Q\rightarrow Q$ satisfying some homotopy conditions on the neighborhoods of the spaces embedded in the Hilbert cube.  He introduced a notion of homotopy among fundamental sequences, setting the shape category of metric compacta as the homotopy classes for this homotopy relation. It is shown that the new category differs only formally from the homotopy category when the space under consideration is an ANR so it is just a way of enlarge the homotopy category with more morphism. For the details, see the original source \cite{Bconcerning}, or the books by Borsuk \cite{Btheory,BtheoryA}.

After Borsuk's description of the shape category for metric compacta, there was a lot of work in Shape, such as different descriptions, extensions to more general spaces (for instance, the important Fox's extension for metric spaces \cite{Fon}), classifications of shape types or shape invariants. As general references, we recommend the books \cite{Btheory,BtheoryA,MSshape,DSshape} and the surveys \cite{Mthirty,Mabsolute}.

As can be easily deduced, Shape is a very simple and intuitive idea, but is very hard to formalize. One of the most useful descriptions of Shape was the inverse system approach, initiated by Mardesic and Segal for compact Hausdorff spaces in \cite{MSshapes}, and further developed by them and some other authors. In this one, Shape is defined using inverse sequences of polyhedra and its inverse limits for topological spaces and, more generally, inverse systems for categories. The best reference for this approach, is the book by the authors \cite{MSshape}.

In this paper, we are concerned with the Multivalued Shape Theory for metric compacta, initiated by Sanjurjo in \cite{San}. The key and acute idea of Multivalued Shape is to replace the shape morphisms by sequences of multivalued maps with decreasing diameters of their images, which is, in some sense, a very natural way of defining them, but also very hard to formalize. By defining a non-trivial sort of homotopic classes in this maps, it is possible to establish a category isomorphic to the shape category of metric compacta. The importance of this theory lies on the fact that it is internal. That is, we do not make use of external elements (such as the Hilbert cube or polyhedra) to describe the morphisms, as in other shape theories. We just use maps between the metric compacta to define the morphisms. We will use a reformulation of the Multivalued Shape here but it is recommended to read the original source \cite{San}, because of its enlighted geometry. The author of the present paper thanks J.M.R. Sanjurjo for his clever ideas and exposition, which had a great impact in the origins of his research. 

\paragraph{\textsc{Multivalued maps and hyperspaces}} This multivalued theory of shape was reinterpreted later by Alonso-Morón and González Gómez in \cite{MGhomotopical}, as it follows. It is based on the observation that multivalued functions are just maps into hyperspaces with the so-called upper semifinite topology. For a compact metric space $X$, we call its \emph{hyperspace} to the set of non-empty closed subsets $$2^X=\left\lbrace C\subset X: C\enspace\text{closed}\right\rbrace$$ and we endow this space with the \emph{upper semifinite topology} $2^X_u$ with base the family $B=\{B(U)\}_{U}$ for all the open subset $U\subset X$, where $B(U)$ is defined as the set of elements of $2^X$ contained in $U$. A deep study of this topology can be found at \cite{MGupper,MGhomotopical}. Even though this space $2^X_u$ has very poor topological properties (is not even Hausdorff, in general), it has a very easy handling. For example, if the space $X$ is discrete, the hyperspace $2^X_u$ is a \textsc{poset} with the order given by $C<D$ iff $C\subset D$. Actually, the upper semifinite topology in the hyperspace is equivalent to the upper semicontinuity in the multivalued maps. We need to define two concepts here. Let $X$ and $Y$ be two compact metric spaces. Consider $2^Y_u$ the hyperspace of $Y$ with the upper semifinite topology. An \emph{approximative map} from $X$ to $Y$ is a sequence of continuous maps $\{f_n\}$, with  $f_n:X\rightarrow 2^Y_U$, such that, for every neighborhood $U$ of the canonical copy of $Y$ in $2^Y_u$, there exists $n_0\in\mathbb{N}$ such that $f_n$ is homotopic to $f_{n+1}$ in $U$ (written $f_n\simeq_U g_n$, meaning there exists a homotopy $H:X\times I\rightarrow U\subset2^Y_u$ between $f_n$ and $f_{n+1}$) for every $n>n_0$. We say that two approximative maps $\{f_n\}$ and $\{g_n\}$ from $X$ to $Y$ are \emph{homotopic} when, for each open neighborhood $U$ of the canonial copy of $Y$ in $2^Y_u$, there exists $n_0\in\mathbb{N}$ such that $f_n$ is homotopic to $g_n$ in $U$ for every $n>n_0$. This way, it is established that, for compact metric spaces, the homotopy classes of approximative maps between two spaces is equivalent to the shape morphisms between them. 

\section{The Main Construction}
We recall the \emph{Main Construction}, introduced in \cite{MCLepsilon} and reformulated in \cite{MMreconstruction}. There, a sequence of finite approximations for a compact metric space $X$, with some special properties, is constructed. Let $\cms$ be a compact metric space. For every $\varepsilon>0$, we can find a finite \emph{$\varepsilon$-approximation} $A\subset X$. That is, for every point $x\in X$, there is a point $a\in A$ with $\text{d}(x,a)<\varepsilon$. An \emph{adjusted approximative sequence} $\left\lbrace\varepsilon_n,A_n\right\rbrace$ consists of a decreasing and tending to zero sequence of positive real numbers $\{\varepsilon_n\}$ and a sequence of finite spaces $\{A_n\}$ such that, for every $\todon$, $A_n$ is an $\varepsilon_n$-approximation of $X$ and $\varepsilon_{n+1}$ is adjusted to it, meaning that, for every $\todon$, $\varepsilon_{n+1}<\frac{\varepsilon_{n}-\gamma_{n}}{2}$, with \[\gamma_n=\text{sup}\{\text{d}(x,A_n):x\in X\}<\varepsilon_n.\] The Main Construction in \cite{MCLepsilon} and Theorem 3 in \cite{MMreconstruction} states the existence of adjusted approximative sequences for every compact metric space. Given an adjusted approximative sequence, we can define a \emph{nearby approximative sequence} $\{q_{A_n}\}$ as the sequence of continuous maps $q_{A_n}:X\rightarrow2^{A_n}_u$, defined by proximity, \[q_{A_n}(x)=\left\lbrace a\in A:\text{d}(x,A_n)=\text{d}(x,a)\right\rbrace,\] and a \emph{finite approximative sequence} (\textsc{fas} for short) $\left\lbrace U_{2\varepsilon_n}(A_{n}),p_{n,n+1}\right\rbrace$, as an inverse sequence with terms \[U_{2\varepsilon_n}(A_{n})=\conjunto{C\in 2^{A_n}_u:\diam(C)<2\varepsilon_n}\subset 2^{A_n}_u\] and continuous maps \[p_{n,n+1}:U_{2\varepsilon_{n+1}}(A_{n+1})\longrightarrow U_{2\varepsilon_{n}}(A_{n})\] with $p_{n,n+1}=q_{A_n}(C)$, for every closed set $C\in U_{2\varepsilon_{n+1}}(A_{n+1})$. Note that the upper semifinite topology in the hyperspace of the discrete space $U_{2\varepsilon_{n}}(A_{n})$ is a finite space with the relation $\subset$ as poset. We shall use some facts about the distance preserved by these maps.
\begin{lem}\label{lem:distmc}
Let $X$ be a compact metric space and $\left\lbrace\varepsilon_n,A_n\right\rbrace$ an adjusted approximative sequence for it. Consider points $x\in X$, $a_n\in A_n$ and $a_m\in A_m$ with $n<m$.
\begin{itemize}[noitemsep]
	\item[i)] If $a_n\in q_{A_n}(x)$ and $a_m\in q_{A_m}(x)$, then $\dist{a_n}{a_m}<\varepsilon_n$.
	\item[ii)] If $a_n\in p_{n,m}(\{a_m\})$, then $\dist{a_n}{a_m}<\varepsilon_n$.
	\item[iii)] If $a_n\in p_{n,m}\cdot q_{A_m}(x)$, then $\dist{a_n}{x}<\varepsilon_n$.
\end{itemize}
\end{lem}
\begin{proof}
For $i)$ we have that \[\text{d}(a_n,a_m)\leqslant\text{d}(a_n,x)+\text{d}(x,a_m)<\gamma_n+\gamma_{n+1}<\gamma_n+\dfrac{\varepsilon_n-\gamma_n}{2}<\varepsilon_n.\] Part $ii)$ is Proposition 2 in \cite{MMreconstruction}. The last claim follows because \[\text{d}(a_n,x)\leqslant\text{d}(a_n,a_m)+\text{d}(a_m,x)<\gamma_n+2\gamma_{n+1}<\gamma_n+2\varepsilon_{n+1}<\varepsilon.\]
\end{proof}
The Main Construction has good properties in terms of reconstruction of the original space. In \cite{MMreconstruction}, it is shown that the inverse limit of every \textsc{fas} has the homotopy type of the original space, since it contains an homeomorphic copy of it which is an strong deformation retract of this limit. In \cite{Mpolyhedral}, several sequences of polyhedra associated to a \textsc{fas} and called \emph{Polyhedral Approximative Sequences (\textsc{pas})} and it is shown that they recover the shape type of the space, conjectured in \cite{MCLepsilon} as the \emph{General Principle}.

\section{Finite type Multivalued Shape} \label{sec:multishape}
In this section we analyze the shape properties of the Main Construction in the realm of the Multivalued Shape Theory for compact metric spaces. The purpose of this relationship is to reflect that the Main Construction captures the shape properties of the space in which is done and that it is able to do it using only finite subsets of the space. We prove some propositions concerning this relationship, in particular two results proposed in \cite{MCLepsilon} and a finiteness theorem for shape morphisms.

We need two technical lemmas about homotopies in hyperspaces with the upper semifinite topology. The advantage of this topology is shown here, since these homotopies are constructed in a very natural way.

\begin{lem}\label{lem:unioncontinua}
Let $X, Y$ be metric compacta and $f,g:X\rightarrow2^Y_u$ two continuous maps. The map $f\cup g:X\rightarrow2^Y_u$ defined by $\left(f\cup g\right)(x)=f(x)\cup g(x)$ is continuous.
\end{lem}
\begin{proof}
For every $x\in X$, the map $f\cup g$ is well defined, because $f(x)$ and $g(x)$ are closed subsets of $Y$, and so is $f(x)\cup g(x)$. Let us take a neighborhood $B(V)$ of $f(x)\cup g(x)$ in $2^Y_u$, where $B(V)=\{C\in2^Y_u : C\subset V\}$ with $V$ an open subset of $Y$ which contains $f(x)\cup g(x)$. The maps $f$ and $g$ are continuous, so there are two neighborhoods of $x$, namely $U_1$ and $U_2$, such that $f(U_1),g(U_2)\subset V$. Then $U=U_1\cap U_2$ is a neighborhood of $x$ such that $(f\cup g)(U)=f(U)\cup g(U)\subset f(U_1)\cup g(U_2)\subset V$, so $f\cup g(U)\subset B(V)$, hence $f\cup g$ is continuous at $x$.
\end{proof} 

\begin{lem}\label{lem:hiperhomotopia}
Let $X, Y$ be metric compacta and let $f,g,h:X\rightarrow2^Y_u$ be continuous maps such that, for every $x\in X$, $f(x), g(x)\subset h(x)$. Then, $f$ and $g$ are homotopic.
\end{lem}
\begin{proof}
The map $H:X\times I\longrightarrow2^X_u$, defined by 
\begin{equation*}
H(x,t) = \left\{
\begin{array}{rl}
f(x) & \text{if } t\in[0,\frac{1}{2}),\\
h(x) & \text{if } t=\frac{1}{2},\\
g(x) & \text{if } t\in(\frac{1}{2},1],
\end{array}\right.
\end{equation*} 
is continuous and hence a homotopy between the two maps. Indeed, $H$ is obviously continuous in every point $(x,t)$ with $t\neq\frac{1}{2}$. For every point $(x,\frac{1}{2})\in X\times\frac{1}{2}$, consider an open neighborhood $B(V)$ of $h(x)$. Because of the continuity of $h$, there is an open neighborhood $U$ of $x$ such that $h(U)\subset B(V)$. But $f(U), g(U)\subset h(U)$, so $U\times I$ is an open neighborhood of $(x,\frac{1}{2})$ such that $H(U\times I)\subset B(V)$, and the continuity of $H$ is shown.
\end{proof}
\begin{obs}
From the previous two lemmas, we can derive that every two maps $f,g:X\rightarrow 2^Y_u$ are homotopic. Despite it seems to be no interest in these trivial homotopies, the point here is to construct them in some subspaces of $2^Y_u$.
\end{obs}

We can show the following result, posed in \cite{MCLepsilon} (as Proposition 21), relating the Main Construction with the Multivalued Shape Theory.
\begin{prop}
Let $X$ be a compact metric space. Consider an adjusted approximative sequence $\{\varepsilon_n, A_n\}$ obtained by performing the Main Construction to $X$. The sequence of maps $\{q_{A_n}\}$ is an approximative map representing the identity shape morphism on $X$.
\end{prop}
\begin{proof} 
Let us first prove that $\{q_{A_n}\}$ is indeed an approximative map. As we know, for each $n\in\mathbb{N}$ the map $q_{A_n}:X\rightarrow 2^Y_u$ is continuous. The family $\{U_\varepsilon\}$ is a base of open neighborhoods of the canonical copy of $X$ inside $2^X_u$ (see \cite{MGhomotopical}), so there exists an $\varepsilon>0$ such that $X\subset U_\varepsilon\subset U$. There exists an integer $n_0$ such that $2\varepsilon_{n_0}<\varepsilon$. We claim that, for every $n\geqslant n_0$, the map $H:X\times I\rightarrow U\subset2^X_u$, defined by
\begin{equation*}
H(x,t) = \left\{
\begin{array}{rl}
q_{A_n}(x) & \text{if } t\in[0,\frac{1}{2}),\\
q_{A_n}(x)\cup q_{A_{n+1}}(x) & \text{if } t=\frac{1}{2},\\
q_{A_{n+1}}(x) & \text{if } t\in(\frac{1}{2},1],
\end{array}\right.
\end{equation*}
is an homotopy between $q_{A_n}$ and $q_{A_{n+1}}$ in $U$. The map is continuous and well defined because, for every $x\in X$, 
\begin{eqnarray*}
q_{A_{n+1}}(x)&\subset&\textrm{B}(x,\varepsilon_{n+1})\subset\textrm{B}(x,\varepsilon_n),\\
q_{A_n}(x)&\subset&\textrm{B}(x,\varepsilon_n),
\end{eqnarray*}
and then
$$\textrm{diam}\left(q_{A_n}(x)\cup q_{A_{n+1}}(x)\right)<2\varepsilon_n<\varepsilon,$$
so the images of the applications are in $U_\varepsilon\subset U$ and it is an homotopy because of Lemma \ref{lem:unioncontinua} and Lemma \ref{lem:hiperhomotopia}.

It is clear that the approximative map $\textrm{id}:X\rightarrow 2^X_u$, with $\textrm{id}(x)=\{x\}$, corresponds to the identity shape morphism of $X$. To prove that $\{q_{A_n}\}_{n\in\mathbb{N}}$ it is homotopic to the identity, we just choose $n_0$ such that $2\varepsilon_{n_0}<\varepsilon$, and use the homotopy (Lemma \ref{lem:unioncontinua} and Lemma \ref{lem:hiperhomotopia}.) $H:X\times I\rightarrow U\subset2^X_u$, defined by
\begin{equation*}
H(x,t) = \left\{
\begin{array}{rl}
q_{A_n}(x) & \text{if } t\in[0,\frac{1}{2}),\\
q_{A_n}(x)\cup \{x\} & \text{if } t=\frac{1}{2},\\
\{x\} & \text{if } t\in(\frac{1}{2},1].
\end{array}\right.
\end{equation*}
\end{proof}

Note that, in the previous result, we obtain an equivalent approximative map with finite images. We can generalize this to any class of approximative map in order to describe a Shape Theory for metric compacta in simpler terms. An approximative map $\{f_n\}_{\todon}$ between compact metric spaces $X$ and $Y$ is said to be of \emph{finite type} if, for every $\todon$ and every $x\in X$, the image $f_n(x)$ is a finite set. For every homotopy class of approximative maps, we can always find a representative of finite type.
\begin{teo}
Let $X, Y$ be compact metric spaces and $\{f_n\}_{\todon}$ an approximative map from $X$ to $Y$. There exists an approximative map $\{\mathfrak{f}_n\}_{\todon}$ of finite type from $X$ to $Y$ homotopic to $\{f_n\}_{\todon}$.
\end{teo}
\begin{proof}
Consider a decreasing sequence of positive real numbers $\{\beta_n\}_{\todon}$ converging to zero and a sequence of finite $\beta_n$-approximations $B_n$ of $Y$. Define, for every $\todon$, the map $\mathfrak{f}_n=r_{B_n}\cdot f_n$, where $r_{B_n}:2^Y_u\rightarrow 2^Y_u$, is the extension of the map $q_{B_n}:Y\rightarrow 2^Y_u$ and hence a continuous map (see \cite{MMreconstruction}, Corollary 1). It is clear that $\mathfrak{f}$ is continuous. Now, we need to show that $\{\mathfrak{f}_n\}_{\todon}$ is an approximative map and it is homotopic to $\{f_n\}_{\todon}$. We are going to prove both statements as consecuences of the following claim: for every open set $U\subset 2^Y_u$ containing the canonical copy of $Y$, there exists $n_0$ such that, for every $n\geqslant n_0$, $f_n\simeq_U \mathfrak{f}_n$. Indeed, let $U\subset 2^Y_u$ such an open set and consider, for every $\todon$, the diameter $D_n$ of the map $f_n$. Since $f_n$ is an approximative map, it is clear that the sequence $\{D_n\}_{\todon}$ converges to zero. The diameter of $\mathfrak{f}_n$ depends on $D_n$, for each $\todon$. For every $x\in X$, and for every two points $y_1, y_2\in f_n(x)$, consider $z_1\in B_n(y_1,\beta_n), z_2\in B_n(y_2,\beta_n)$. Then, we have
$$\dist{z_1}{z_2}\leqslant\dist{z_1}{y_1}+\dist{y_1}{y_2}+\dist{y_2}{z_2}<2\beta_n+D_n$$
and hence $\diam\left(\mathfrak{f}_n\right)<\beta_n+D_n$. Now, let $\varepsilon>0$ be a real number such that $U_{\varepsilon}\subset U$, and select $n_0$ such that $2\beta_n+D_n<\varepsilon$ for every $n\geqslant n_0$. Then, the map $H:X\times I\rightarrow U$, defined by 
\begin{equation*}
H(x,t) = \left\{
\begin{array}{rl}
f_n(x) & \text{if } t\in[0,\frac{1}{2}),\\
f_n(x)\cup \mathfrak{f}_n(x) & \text{if } t=\frac{1}{2},\\
\mathfrak{f}_n(x) & \text{if } t\in(\frac{1}{2},1],
\end{array}\right.
\end{equation*}
is continuous by Lemma \ref{lem:unioncontinua} and Lemma \ref{lem:hiperhomotopia}, and hence a homotopy between $f_n$ and $\mathfrak{f}_n$ in $U$. Moreover, being $\{f_n\}_{\todon}$ an approximative map, there exists $m_0$ such that, for every $n\geqslant m_0$, $f_n$ is homotopic to $f_{n+1}$ in $U$. Finally, for $n>\max{n_0,m_0}$, we have $$\mathfrak{f}_n\simeq_U f_n\simeq_U f_{n+1}\simeq_U \mathfrak{f}_{n+1},$$ which shows the two statements claimed to finish the proof.
\end{proof}

We finish this section by showing this result, also proposed in \cite{MCLepsilon} (Proposition 20), which establishes a deeper connection of the Main Construction with the shape of the space, because it takes into account the maps $p_{n,n+1}$ between the finite spaces.
\begin{prop}\label{teo:commuptohom}Let $X$ be a compact metric space and consider the sequences obtained with the Main Construction over $X$. The following diagram is commutative, up to homotopy, for every $n\in\mathbb{N}$.
$$\xymatrix@C=2cm@R=1.5cm{X\ar[r]^{id} \ar[d]_{q_{A_{n+1}}} & X \ar[d]^{q_{A_n}}\\
U_{2\epsilon_{n+1}}(A_{n+1})\ar[r]_{p_{n+1,n}} & U_{2\epsilon_n}(A_n).}$$
\end{prop}
\begin{proof}
To prove this commutativity we need a homotopy between the maps $p_{n,n+1}\cdot q_{A_{n+1}}$ and $q_{A_n}$. Such a homotopy is given by $H:X\times I\rightarrow U_{2\varepsilon_n}(A_n)$, with
\begin{equation*}
H(x,t) = \left\{
\begin{array}{rl}
q_{A_n}(x) & \text{if } t\in[0,\frac{1}{2}),\\
q_{A_n}(x)\cup p_{n,n+1}\cdot q_{A_{n+1}}(x) & \text{if } t=\frac{1}{2},\\
p_{n,n+1}\cdot q_{A_{n+1}}(x) & \text{if } t\in(\frac{1}{2},1].
\end{array}\right.
\end{equation*}
This is a homotopy because, by Lemma \ref{lem:distmc}, we have that \[\textrm{diam}\left(q_{A_n}(x)\cup p_{n,n+1}\cdot q_{A_{n+1}}(x)\right)<2\varepsilon_n,\] and hence it is well defined and using Lemma \ref{lem:unioncontinua} and Lemma \ref{lem:hiperhomotopia} we conclude that it is a homotopy between continuous maps and we are done.
\end{proof}

\paragraph{Acknowledgments} The author is grateful to his thesis advisor, M.A. Morón, for his guidance on the topic of the present article.
\paragraph{Funding} This work has been partially supported by the research project PGC2018-098321-B-I00(MICINN). The author has been also supported by the FPI Grant BES-2010-033740 of the project MTM2009-07030 (MICINN).

\addcontentsline{toc}{chapter}{References}
\bibliographystyle{siam}%Choose a bibliograhpic style
\bibliography{bibFinmul}

\vspace{1cm}
%author1
\noindent\textsc{Diego Mondéjar}\\
Departamento de Métodos Cuantitativos, CUNEF Universidad, 28040 Madrid, Spain\\
\texttt{diego.mondejar@cunef.edu}

\end{document}